\documentclass[10pt]{article}

\usepackage[latin9]{inputenc}
\usepackage{latexsym, bm, amsmath, amssymb, graphics, amsthm, dsfont, enumerate, physics} %
\usepackage[a4paper,
            bindingoffset=0.2in,
            left=0.8in,
            right=0.8in,
            top=0.8in,
           bottom=0.8in,
            footskip=.25in]{geometry}
\usepackage{xcolor,placeins}
\usepackage[ruled,vlined]{algorithm2e}
\usepackage{tikz}

\usetikzlibrary{calc,trees,positioning,arrows,chains,shapes.geometric,  decorations.pathreplacing,decorations.pathmorphing,shapes,  matrix,shapes.symbols}
\usepackage{tikz-cd}
\tikzcdset{scale cd/.style={every label/.append style={scale=#1},
    cells={nodes={scale=#1}}}}

\usepackage{amsmath} 
\usepackage{amsthm} 

\usepackage{hyperref} 
\usepackage{multirow,multicol}

\usetikzlibrary{arrows.meta,
                positioning,
                shapes}
\newtheorem{theorem}{Theorem}

\newtheorem{lemma}{Lemma}

\newtheorem{Def}{Definition}

\newtheorem{Ex}{Example}

\usepackage{comment}

\usepackage{booktabs}

\usepackage{setspace}

\usepackage{forest}          
\usetikzlibrary{arrows.meta}

\usepackage{xcolor}

\usepackage{tcolorbox}

\usepackage{svg}
\usepackage{subcaption}

\usepackage{enumitem}

\begin{document}

\title{Conditional Probability Spaces and the Structure of Agreement\footnote{We thank Emiliano Catonini and Stuart Zoble for extremely valuable feedback, Emiliano Catonini for suggesting Theorem 5, and a referee for a very important observation on a previous version.  E.Y.~acknowledges financial support from National Natural Science Foundation of China (NSFC), Award No.~72403259. A.B.~acknowledges financial support from NYU Stern School of Business, NYU Shanghai, and J.P. Valles. }
}

\author
{Erya Yang \footnote{Lingnan College, Laboratory of Mezzoeconomics and Regional Industrial Coordinated Development, Shenzhen Institute of Economics, Sun Yat-sen University, China; yangery@mail.sysu.edu.cn; corresponding author}
\and
{Adam Brandenburger \footnote{Stern School of Business, Tandon School of Engineering, NYU Shanghai, New York University, New York, NY 10012, U.S.A.; adam.brandenburger@stern.nyu.edu}}
    }

\date{\today}
\maketitle

\begin{abstract}
We use the machinery of a conditional probability space (R\'enyi, 1955) to obtain an Agreement Theorem (Aumann, 1976) under general conditions.  A conditional probability space (CPS) is a family of probability measures defined relative to a family of conditioning events that satisfies concentration and a chain rule.  Using this apparatus, we derive an Agreement Theorem that dispenses with the traditional assumptions of a common prior, information partitions, positivity of measure, and knowledge operators.  Our treatment can be viewed as ``deconstructing" the classic Agreement Theorem, by showing how it can be built up from local probabilistic-epistemic ingredients.  The main technical contribution is to define an augmentation procedure for CPSs that adds into the conditioning family all (sub)events that receive probability $1$ -- thereby achieving consistency between an agent's information and subjective certainty of events.
\end{abstract}

\thispagestyle{empty}
\section{Introduction} \label{sec:1}
The appearance of unbounded measures in applications of probability theory including Bayesian statistics, statistical and quantum mechanics, and probabilistic number theory motivated R\'enyi (1955) to introduce the concept of a conditional probability space (CPS).  For a CPS, one begins with a measurable space together with a family of nonempty conditioning events and associates with each such event a probability measure interpreted as the probability conditional on that event.  If the conditioning events are finite or compact or controlled in some other way, the problem of unbounded measures can be avoided.  R\'enyi developed his theory in a series of publications; see R\'enyi (1955, 1956, 1970a; 1970b).

CPSs have also proved well suited to game-theoretic applications.  An extensive game comes with natural conditioning events, namely, a player's information sets.  A CPS specifies the player's beliefs that are compatible with being at the information set in question.  Early work in this direction was undertaken by Myerson (1986), while Battigalli and Siniscalchi (1999, 2002) developed the fundamental epistemic game theory of trees using CPSs.

In this paper, we use the CPS framework to revisit the classic Agreement Theorem due to Aumann (1976).  The theorem states that if two agents share a common prior, update beliefs according to their respective information partitions, and find that their posterior probability assessments of a shared event of interest are common knowledge, then the two assessments must coincide.  The result is a cornerstone of interactive epistemology and has been applied in areas such as equilibrium analysis (Aumann and Brandenburger, 1995; Pacuit and Roy, 2025) and market efficiency (Milgrom and Stokey, 1982; Sebenius and Geanakoplos, 1983; Gizatulina and Hellman, 2019).  

The classical formulation of the theorem relies on several structural assumptions: (i) the existence of a common prior, (ii) the use of information partitions by the agents, (iii) a positivity condition on the meet of the partitions or on partition cells, and (iv) the use of a knowledge modality defined relative to the partitions.  Over the years, the literature has explored various relaxations of these assumptions.  Bonanno and Nehring (1999) and Hellman (2013) investigate weakenings of the common-prior assumption.  Geanakoplos (1989, 2021) and Samet (1990, 2022) generalize the informational structure beyond partitions.  Nielsen (1984) and Brandenburger and Dekel (1987) study settings with null events.  Brandenburger and Dekel (1987) and Monderer and Samet (1989) analyze the weaker certainty modality, while Heifetz, Meier, and Schipper (2013) study the unawareness modality and agreement.  Tsakas (2018) employs a CPS for each agent, which, on positive-probability conditioning events, is derived from a common prior.  Bach and Perea (2013) and Bach and Cabessa (2023) use a lexicographic probability system (Blume, Brandenburger, and Dekel, 1991) as an extended common prior.

This paper shows that it is possible to provide a unified treatment of several of these issues within the CPS framework.  The CPS setting allows us to separate the structural features of probabilistic reasoning that are essential for agreement from those that are conveniences.  The two defining properties of CPSs play a central role.  The first is concentration, which means that each conditional probability assigns probability $1$ to its conditioning event, and the second is the chain rule, which relates probabilities of nested conditioning events.  Our state space is finite and we assume each agent's family of conditioning events is closed under unions and nonempty intersections and covers the space.  Notice that we do not assume closure under complementation, so we generalize beyond information partitions.

We place three assumptions on the agents' CPSs, of which two are intra-agent and one is inter-agent.  The first intra-agent assumption is certainty reflection, which says that if, at some state, an agent assigns a particular probability to some event, then the agent is certain of that assessment.  (In S4 modal logic for certainty, this assumption for the probability-$1$ case is called positive introspection.)  The second intra-agent assumption is what we call $1$-closedness.  On some of an agent's conditioning events, the agent may assign probability $1$ to -- that is, be certain of -- strict subsets of those events.  This is not ruled out in the definition of a CPS.  But it creates a conceptual gap between an agent's information and (subjective) certainty.  We call a CPS $1$-closed if all such probability-$1$ events are (already) members of the agent's conditioning family.  The inter-agent assumption is local consistency, which says that conditional on any event that lies in both Alice's and Bob's conditioning families, their conditional probabilities must be the same.  On any other event, their probabilities can differ.  This is our relaxation of the common prior.

Our main result is: \textit{Assume that Alice and Bob possess CPSs and that certainty reflection, $1$-closedness, and local consistency hold.  Fix an event $E$ of interest.  Suppose, at a state $\omega$, it is common certainty that Alice assigns conditional probability $q_A$ to $E$ and Bob assigns conditional probability $q_B$ to $E$.  Then $q_A = q_B$.}

Stated less formally, our theorem says that disagreement under common certainty cannot survive once four things are true: (i) each agent reasons consistently via conditional probabilities, (ii) agents are certain what their own beliefs are, (iii) any sub-event an agent treats as certain is incorporated into that agent's informational structure, (iv) whenever the agents are looking at the same situation, they form the same probabilities there.

A second goal of the paper is a ``deconstruction" of the Agreement Theorem into arguably more basic components.  In this vein, a notable aspect of our probabilistic-epistemic ingredients (i)-(iv) is that they have an if-then character.  This is built into the definition of a CPS, which, in place of a single probability measure, is made up of a family of probability measures that say ``If the agent is informed of a particular event, then the agent would have the following probabilities."  Certainty reflection also has an if-then character, saying ``If the agent holds particular probabilities, then the agent is certain of this."  The $1$-closedness property is an internal coherency requirement on if-then statements.  Local consistency says ``If the two agents happen to have the same information, then they hold the same probabilities."  These if-then ingredients have a local and arguably more basic flavor than does the common prior.  We do not claim that we have identified a minimal architecture of agreement (we have no definition of minimality), but we do suggest that our approach may shed some light on what seems fundamentally involved in arriving at an Agreement Theorem.

A technical device we introduce is what we call $1$-augmentation of a family of conditioning events.  If an agent assigns conditional probability $1$ to an event inside some conditioning event, then that event already behaves like information from the agent's perspective.  Our $1$-augmentation procedure therefore enlarges the conditioning family by promoting all such probability-$1$ events to conditioning events themselves.  This construction is inspired by the ``posterior completion" process in Brandenburger and Dekel (1987, Definition 2.3), which itself is analogous to the augmentation of $\sigma$-fields in probability theory (Chung, 1974, Chapter 2, Exercise 20).  Using the R\'enyi (1956) representation theorem for CPSs on conditioning families closed under finite unions, we show that a CPS can be extended to its $1$-augmented conditioning family.  Finitely many iterations of this operation stabilize are at a $1$-closed CPS.

The paper proceeds as follows.  Section \ref{sec:2} introduces the CPS framework, the definition of common certainty, and an example of common certainty of disagreement in the absence of our conditions.  Section \ref{sec:3} proves our Agreement Theorem under our conditions and gives an example showing that these conditions do not trivialize the framework.  Section \ref{sec:4} defines our $1$-augmentation procedure, and states the results on extension of CPSs and finite stabilization of $1$-augmentation.  Section \ref{sec:5} discusses the relationship between certainty and knowledge and comments on open questions, including the infinite case.  The Appendix contains the proofs of the results in Section \ref{sec:4}.

\section{Framework} \label{sec:2}
Fix a finite state space $\Omega$ with the power set $\mathcal F = 2^\Omega$ of events.

\begin{Def} \label{def:cps}
A \textbf{conditional probability space} (\textbf{CPS}) is a pair $(\mathcal G, p)$, where $\mathcal G \subseteq \mathcal F$ is a family of nonempty conditioning events, and $p$ associates with each $G \in \mathcal G$ a probability measure $p_G$ on $\Omega$, satisfying:
\begin{enumerate} [label=(\roman*)]
\item $p_G(G) = 1$ for every $G \in \cal G$;
\item $p_G(E) = p_G(F) \, p_F(E)$ for every $E \subseteq F \subseteq G$ with $E \in \cal F$ and $F, G \in \cal G$.
\end{enumerate}
\end{Def}

The CPS concept was introduced by R\'enyi (1955) for the general measurable case.  We specialize it here to the finite setting (and comment on the infinite extension later).  The first condition is concentration and the second condition is the chain rule applied to nested events with no positivity requirements.  We will also refer to $p$ alone as a CPS, when $\mathcal G$ has already been specified.

In our application, we suppose Alice has a CPS $(\mathcal G_A, p^A)$.  We take $\mathcal G_A$ to be closed under unions and nonempty intersections and to cover $\Omega$, where the latter means that, for each $\omega \in \Omega$, there is a $G \in {\cal G}_A$ with $\omega \in G$.  Likewise, Bob has a CPS $(\mathcal G_B, p^B)$ obeying the same closure and covering conditions.

Given a state $\omega \in \Omega$, let $m_A(\omega)$ be the $\mathcal G_A$-atom containing $\omega$:
\begin{equation}
m_A(\omega) = \bigcap_{\omega \in G \in \mathcal G_A} \, G.
\end{equation}
Covering implies $m_A(\omega) \not= \emptyset$ and closure under intersections implies $m_A(\omega) \in \mathcal G_A$.

\begin{Def} \label{def:certainty}
Fix an event $E$ and a state $\omega$.  Say Alice is \textbf{certain of} $E$ at $\omega$ if $p^A_{m_A(\omega)}(E) = 1$.
\end{Def}

We write the event that Alice is certain of $E$ as:
\begin{equation} \label{C-operator}
C_A(E) = \{\omega \in \Omega \, : \, p^A_{m_A(\omega)}(E) = 1\}.
\end{equation}
Define certainty for Bob and the event $C_B(E)$ analogously.  Next, fix numbers $q_A, q_B \in [0, 1]$ and define recursively:
\begin{align}
A^0 = \bigl\{ \omega \in \Omega : p_{m_A(\omega)}^A(E) = q_A \bigr\}&, \,\, 
B^0 = \bigl\{ \omega \in \Omega : p_{m_B(\omega)}^B(E) = q_B \bigr\}, \label{eq-4} \\
A^{n+1} = A^n \cap C_A(B^n)&, \,\, B^{n+1} = B^n \cap C_B(A^n), \label{eq-5}
\end{align}
for $n \ge 0$.  Let:
\begin{equation}
C^\infty = \bigcap_{n=0}^\infty A^n \cap \bigcap_{n=0}^\infty B^n.
\end{equation}

\begin{Def} \label{def:common-certainty}
Fix an event $E$ and a state $\omega$.  Say it is \textbf{common certainty at} $\omega$ that Alice assigns probability $q_A$ to $E$ and Bob assigns probability $q_B$ to $E$ if $\omega \in C^\infty$.
\end{Def}

In the special case $q_A = q_B  = 1$, we simply say that the event $E$ is common certainty at $\omega$.  We next make an intra-agent assumption and an inter-agent assumption.

\begin{Def} \label{def:reflection}
Fix an agent $i = A, B$.  We say $i$ satisfies \textbf{certainty reflection} if for every event $E$, every number $q \in [0,1]$, and every state $\omega$:
\begin{equation}
p^i_{m_i(\omega)}(E) = q  \,\,\, \text{implies} \,\,\, p^i_{m_i(\omega)}\bigl(\{\omega^\prime \in \Omega : p^i_{m_i(\omega^\prime)}(E) = q\}\bigr) = 1.
\end{equation}
\end{Def}

In words, this says that if, at some state $\omega$, agent $i$ assigns probability $q$ to an event $E$, then, at $\omega$, the agent is certain of assigning probability $q$ to $E$.  In brief, there is certainty of own belief.  As noted in the Introduction, this condition is analogous to positive introspection in S4 modal logic for certainty.  Theorem \ref{thm:reflection} later provides a characterization of certainty reflection in terms of the behavior of the $\mathcal G_i$-atoms. 

To state our inter-agent assumption, let $\mathcal G_A \wedge \mathcal G_B$ denote the meet (intersection) of the families $\mathcal G_A$ and $\mathcal G_B$, and let $m(\omega)$ be the $(\mathcal G_A \wedge \mathcal G_B)$-atom containing $\omega$.  It follows from our assumptions on $\mathcal G_A$ and $\mathcal G_B$ that $\Omega \in \mathcal G_A \wedge \mathcal G_B$ and so $m(\omega) \not= \emptyset$.

\begin{Def} \label{def:local-consistency}
Alice and Bob satisfy \textbf{local consistency} at a state $\omega \in \Omega$ if:
\begin{equation}
p^A_{m(\omega)} = p^B_{m(\omega)}.
\end{equation}
\end{Def}

Note that this is a local condition at a given state and therefore quite different from the common prior, which is a global condition over the entire state space.  Moreover, the condition is required to hold only for the atom at $\omega$ of the family of events that are common to both Alice and Bob.

The example to follow establishes that, in our more general setting, the Agreement Theorem fails without further conditions.  (We maintain a common prior in the sense that $p^A_\Omega = p^B_\Omega$ and employ information partitions, to sharpen the failure of the Agreement Theorem.  We will illustrate the full generality of our framework later.)

\begin{Ex} \label{ex:disagree-cc}
Let $\Omega = \{a, b, c, d\}$ and set:
\begin{align}
\mathcal G_A &= \big\{\{a, d\}, \{b, c\}, \Omega \big\}, \\
\mathcal G_B &= \big\{\{a, b, c\}, \{d\}, \Omega \big\}.
\end{align}
Define CPSs as follows:
\begin{align}
p^A_{\Omega} &= \delta_d, \,\, p^A_{\{a, d\}} = \delta_d, \,\, p^A_{\{b, c\}} = \delta_b, \\
p^B_{\Omega} &= \delta_d, \,\, p^B_{\{d\}} = \delta_d, \,\, p^B_{\{a, b, c\}} = \delta_c,
\end{align}
where $\delta_\omega$ is the Dirac measure on $\{\omega\}$.  These are indeed CPSs, since the only non-vacuous chain-rule restriction is:
\begin{equation}
p^A_\Omega(\{d\}) = p^A_\Omega(\{a, d\}) \, p^A_{\{a, d\}}(\{d\}),
\end{equation}
which is satisfied.  Certainty reflection is automatic because the atoms of each conditioning family form a partition.  Local consistency holds at every state since $\mathcal G_A \wedge \mathcal G_B = \{\Omega\}$, so that $m(\omega) = \Omega$ for all $\omega$, and we have $p^A_\Omega = p^B_\Omega$.

Let $E = \{b\}$ be the event of interest, $q_A = 1$, $q_B = 0$, and calculate:
\begin{align}
A^0 = \{b, c\}&, \,\,\,\, B^0 = \Omega, \\
C_A(B^0) = \Omega&, \,\,\,\, C_B(A^0) = \{a, b, c\}, \\
A^1 = A^0 \cap C_A(B^0) = \{b, c\},& \,\,\,\, B^1 = B^0 \cap C_B(A^0) = \{a, b, c\}, \\
C_A(B^1) = \{b, c\}&, \,\,\,\,\, C_B(A^1) = \{a, b, c\}, \\
A_2 = A^1 \cap C_A(B^1) = \{b, c\}&, \,\,\,\, B_2 = \{a, b, c\},
\end{align}
from which the recursion stabilizes at:
\begin{equation}
A^n = \{b, c\} \,\, \text{and} \,\, B^n = \{a, b, c\} \,\, \text{for all} \,\, n.
\end{equation}

It follows that $C^\infty = \{b, c\}$.  At state $b$ or $c$, it is common certainty that Alice assigns probability $1$ to $E$ and Bob assigns probability $0$ to $E$.  There is common certainty of disagreement. 
\end{Ex}

What drives common certainty of disagreement in this example is a conceptual gap between the agents' information and their probabilities.  Each agent's CPS encodes ``sharper" certainty judgments than their conditioning events express.  Thus, Alice is certain of $\{b\}$ when conditioning on $\{b, c\}$, and Bob is certain of $\{c\}$ when conditioning on $\{a, b, c\}$.  While Alice is informed only of the event $\{b, c\}$, her probability assessment is as if she knows the finer event $\{b\}$.  Analogously for Bob.  In the next section, we close this gap.

\section{Main Result} \label{sec:3}
From here on, we hold fixed the assumption that conditioning families are closed under unions and nonempty intersections and cover $\Omega$.  Fix such a family $\mathcal G$ and consider the family of sets:
\begin{equation} \label{eq:augment}
\mathcal E = \bigl\{E \subseteq \Omega : p_G(E) = 1 \,\, \text{for some} \,\, G \in \mathcal G \,\, \text{with} \,\, E \subseteq G\bigr\}.
\end{equation}
The family $\mathcal E$ consists of those events that are contained in members of the agent's conditioning family and to which the agent assigns (conditional) probability $1$.  Since the agent is certain of these events, it makes sense to put them on an equal footing with the conditioning events themselves.  This motivates the following definition.

\begin{Def} \label{def:1-closed}
Let $(\mathcal G, p)$ be a CPS.  We say that $(\mathcal G, p)$, or simply $\mathcal G$ when $p$ is understood, is $\mathbf{1}$\textbf{-closed} if $\mathcal E = \mathcal G$.
\end{Def}

Note that, by the concentration property of a CPS, we always have $\mathcal G \subseteq \mathcal E$.  So, the force of Definition \ref{def:1-closed} is that $\mathcal E \subseteq \mathcal G$.  This says exactly that the family $\mathcal G$ is closed in the sense that all events that should be treated as conditioning events are (already) included in it.

The concept of $1$-closedness is related to the augmentation operation of probability theory (Chung, 1974, Chapter 2, Exercise 20) that adds in events $E$ to a sub-$\sigma$-field $\mathcal G$ with $p(E \, \Delta \, G) = 0$ for some $G \in \mathcal G$.  The key differences are: (i) we work with CPSs not probability measures, (ii) we do not consider events $E$ with $p_G(E) = 0$ for some $G \in \mathcal G$, and (iii) we operate with a structure that is already assumed to contain the ``additional" events.  In Section \ref{sec:4}, we examine the actual process of augmentation in our setting.  The step to a conditional notion of augmentation was taken earlier by Brandenburger and Dekel (1987, Definition 2.3), working with proper regular conditional probabilities (Blackwell and Ryll-Nardzewski, 1963; Blackwell and Dubins, 1975).

We now work towards our Agreement Theorem via a series of lemmas.

\begin{lemma} \label{lem:certainty-knowledge}
Fix an agent $i = A, B$, and suppose $(\mathcal G_i, p^i)$ is $1$-closed.  Consider an event $E$ and a state $\omega$.  If $\omega\in E$ and $p^i_{m_i(\omega)}(E) = 1$, then $m_i(\omega) \subseteq E$.
\end{lemma}

\begin{proof}
Write $m = m_i(\omega)$ and assume $\omega \in E$ and $p^i_m(E) = 1$.  Let $H = m \cap E$ and use concentration to get:
\begin{equation}
p^i_m(H) = p^i_m( m \cap E) = 1.
\end{equation}
Since $(\mathcal G_i, p^i)$ is $1$-closed, every subset of a conditioning event that has conditional probability $1$ is itself in $\mathcal G_i$. Therefore $H \in \mathcal G_i$.  Since $m$ is the $\mathcal G_i$-atom containing $\omega$, no member of $\mathcal G_i$ containing $\omega$ can be a strict subset of $m$, so that:
\begin{equation}
H = m \cap E \supseteq m,
\end{equation}
and therefore $m \subseteq E$.
\end{proof}

\begin{Def} \label{def:saturation}
Fix $i = A, B$ and an event $E$.  We say that $E$ is $\boldsymbol{i}$-\textbf{saturated} if $\omega \in E$ implies $m_i(\omega)\subseteq E$.
\end{Def}

\begin{lemma} \label{lem:belong}
Fix $i = A, B$ and a nonempty event $E$.  If $E$ is $i$-saturated, then it lies in $\mathcal G_i$.
\end{lemma}

\begin{proof}
If $E$ is $i$-saturated, then:
\begin{equation} \label{eq:saturated}
E = \bigcup_{\omega \in E} m_i(\omega).
\end{equation}
Each atom $m_i(\omega) \in \mathcal G_i$.  Since $\mathcal G_i$
is closed under unions, we conclude that $E \in \mathcal G_i$.
\end{proof}

\begin{lemma} \label{lem:saturate-2}
Suppose the CPSs $(\mathcal G_A, p^A)$ and $(\mathcal G_B, p^B)$ satisfy certainty reflection and are $1$-closed.  Then $A^n$ is $A$-saturated and $B^n$ is $B$-saturated for all $n \ge 0$. Moreover, the set $C^\infty$ is $A$-saturated and $B$-saturated.
\end{lemma}

\begin{proof}
If $\omega \in A^0$, then certainty reflection implies:
\begin{equation}
p^A_{m_A(\omega)}(A^0) = 1.
\end{equation}
Since $\omega \in A^0$, Lemma \ref{lem:certainty-knowledge} implies $m_A(\omega) \subseteq A^0$, as required.  The argument for $B^0$ is symmetric.

For any event $F$, the set $C_A(F)$ is $A$-saturated.  Indeed, if $\omega \in C_A(F)$, then:
\begin{equation}
p^A_{m_A(\omega)}(F) = 1,
\end{equation}
and certainty reflection (with $q_A = 1$) gives:
\begin{equation}
p^A_{m_A(\omega)}(C_A(F)) = 1.
\end{equation}
Since $\omega \in C_A(F)$, Lemma \ref{lem:certainty-knowledge} implies $m_A(\omega) \subseteq C_A(F)$.

Now proceed by induction.  If $A^n$ is $A$-saturated and $C_A(B^n)$ is $A$-saturated, then it is immediate that their intersection $A^{n+1} = A^n \cap C_A(B^n)$ is $A$-saturated.  The argument for $B^n$ is symmetric.

Let $x \in C^\infty$ and $y \in m_A(x)$.  Since each $A^n$ is $A$-saturated and $x \in A^n$ for all $n$, we have $y \in A^n$ for all $n$.  Also, $x \in A^{n+1}$ implies $x \in C_A(B^n)$, and so:
\begin{equation}
p^A_{m_A(x)}(B^n) = 1.
\end{equation}
Since $x \in B^n$, Lemma \ref{lem:certainty-knowledge} gives $m_A(x) \subseteq B^n$, from which $y \in B^n$ for all $n$.  Therefore, $y \in A^n \cap B^n$ for all $n$, and so $y \in C^\infty$.  We have now shown that $m_A(x) \subseteq C^\infty$.  The proof for $m_B(x)$ is symmetric.
\end{proof}

\begin{lemma} \label{lem:contain}
If $\omega \in C^\infty$, then $m(\omega) \subseteq C^\infty$.
\end{lemma}

\begin{proof}
By the previous result, the set $C^\infty$ is both $A$-saturated and $B$-saturated.  Since $\omega \in C^\infty$, the set $C^\infty$ is nonempty.  Using Lemma \ref{lem:belong}, it follows that $C^\infty$ lies in $\mathcal G_A \wedge \mathcal G_B$.  Since $m(\omega)$ is the smallest such set containing $\omega$, we conclude that $m(\omega) \subseteq C^\infty$.
\end{proof}

The next two lemmas lead to the key averaging argument over unions of possibly non-disjoint atoms that is needed for our Agreement Theorem.  The strategy in Lemma \ref{lem:ordered} is analogous to a proof step in Theorem 7 of Samet (1990), which averages over non-disjoint atoms of knowledge structures.  A knowledge-based analog to our Lemma \ref{lem:average-atoms} is postulated as a property (the generalized sure-thing principle) in Geanakoplos (1989, 2021).  We return to the certainty-knowledge relationship in Section \ref{sec:5}. 

\begin{lemma} \label{lem:ordered}
Fix $i = A, B$ and a set $G$ in $\mathcal G_i$.  Let $\mu_1, \dots, \mu_K$ be the distinct $\mathcal G_i$-atoms contained in $G$, ordered so that $\mu_r \subsetneq \mu_s$ implies $r < s$.  For each $j = 1, \dots, K$, let $ U_j = \bigcup_{k=1}^j \mu_k$.  Then:
\begin{enumerate}  [label=(\roman*)]
\item $G = \bigcup_{k=1}^K \mu_k$;
\item for every $j = 2, \dots, K$:
\begin{equation} \label{eq:ordered}
U_{j-1} \cap \mu_j = \bigcup_{\{r < j : \mu_r \subsetneq \mu_j\}} \mu_r.
\end{equation}
\end{enumerate}
\end{lemma}

\begin{proof}
For Part (i), it is immediate that $G$ is $i$-saturated.  The conclusion follows by writing $G$ in the form of Equation \ref{eq:saturated}.  (This does not use the ordering of atoms.)

For Part (ii), fix some $2 \le j \le K$.  We first show that the left-hand side of Equation \ref{eq:ordered} is contained in the right-hand side.  Fix $x \in U_{j-1} \cap \mu_j$.  From $x \in U_{j-1}$, there is some $\ell < j$ such that $x \in \mu_\ell$.  Because $\mu_j, \mu_\ell \in \mathcal G_i$ and both contain $x$, the atom $m_i(x)$ satisfies $m_i(x) \subseteq \mu_j \cap \mu_\ell$.  Since $x \in G$, Part (i) implies that $m_i(x) = \mu_r$ for some $1 \le r \le K$.  This implies $\mu_r \subseteq \mu_j$.  We now claim that $r < j$.  Indeed, if $r > j$, then, since the atoms are distinct, we get $\mu_r \subsetneq \mu_j$ and so $r < j$, a contradiction.  Suppose $r = j$.  Then $\mu_j = m_i(x) \subseteq \mu_\ell$.  Since $\ell \not= j$ and the atoms are distinct, this inclusion must be strict.  From the ordering of the atoms, we then get $j < \ell$, a contradiction.  This implies that $r < j$ and therefore, again using distinctness, $\mu_r \subsetneq \mu_j$.  We have now shown that $x \in \mu_r$ for some $r < j$ with $\mu_r \subsetneq \mu_j$, establishing the claimed inclusion in Equation \ref{eq:ordered}.

For the reverse inclusion, fix $r < j$ and suppose $\mu_r \subsetneq \mu_j$.  Then $\mu_r \subseteq U_{j-1}$ by the definition of $U_{j-1}$ and therefore $\mu_r \subseteq U_{j-1} \cap \mu_j$.  Taking the union over all such $r$ establishes that the right-hand side of Equation \ref{eq:ordered} is contained in the left-hand side, completing the proof.
\end{proof}

Note that Part (ii) of Lemma \ref{lem:ordered} says that $U_{j-1} \cap \mu_j$ is either a union of atoms chosen from $\{\mu_1, \dots, \mu_{j-1}\}$ or empty.

\begin{lemma} \label{lem:average-atoms}
Fix $i = A, B$ and a set $G$ in $\mathcal G_i$.  Let $\mu_1, \dots, \mu_K$ be the distinct $\mathcal G_i$-atoms contained in $G$.  Then, if $p^i_{\mu_j}(E) = q_i$ for all $j = 1, \dots, K$, we have $p^i_G(E) = q_i$.
\end{lemma}

\begin{proof}
The proof is by induction on $K$.  If $K = 1$, then $G = \mu_1$ by Part (i) of Lemma \ref{lem:ordered} and there is nothing to prove.  Now suppose the result holds whenever a member of $\mathcal G_i$ contains at most $K - 1$ distinct $\mathcal G_i$-atoms.  Let $G \in \mathcal G_i$ contain exactly $K$ such atoms, say $\mu_1, \dots, \mu_K$, and suppose that $p^i_{\mu_j}(E) = q_i$ for every $j = 1, \dots, K$.  Order the atoms and define $U_{K-1}$ as in Lemma \ref{lem:ordered}.  Applying Part (i) of Lemma \ref{lem:ordered}, we can write:
\begin{equation} \label{eq:average}
G = U_{K-1} \cup \mu_K.
\end{equation}
By closure under unions, $U_{K-1}$ lies in $\mathcal G_i$, and so the induction hypothesis gives $p^i_{U_{K-1}}(E) = q_i$.  Also, by assumption, we have $p^i_{\mu_K}(E) = q_i$.

Next apply Part (ii) of Lemma \ref{lem:ordered} with $j = K$.  If $U_{K-1} \cap \mu_K$ is nonempty, it is a union of some of the atoms
$\mu_1, \dots, \mu_{K-1}$ and therefore, by closure under unions, lies in $\mathcal G_i$.  The induction hypothesis gives:
\begin{equation}
p^i_{U_{K-1} \cap \mu_K}(E) = q_i.
\end{equation}
Now use the chain rule for CPSs to write:
\begin{align}
p^i_G(E \cap U_{K-1}) &= p^i_G(U_{K-1}) \times p^i_{U_{K-1}}(E), \\
p^i_G(E \cap \mu_K) &= p^i_G(\mu_K) \times p^i_{\mu_K}(E), \\
p^i_G(E \cap U_{K-1} \cap \mu_K) &= p^i_G(U_{K-1} \cap \mu_K) \times p^i_{U_{K-1} \cap \mu_K}(E). \label{eq:third-1}
\end{align}
Substituting the value $q_i$ into each of these identities yields:
\begin{align}
p^i_G(E \cap U_{K-1}) &= q_i \times p^i_G(U_{K-1}), \\
p^i_G(E \cap \mu_K) &= q_i \times p^i_G(\mu_K), \\
p^i_G(E \cap U_{K-1} \cap \mu_K) &= q_i \times p^i_G(U_{K-1} \cap \mu_K). \label{eq:third-2}
\end{align}
Using Equation \ref{eq:average}, the inclusion-exclusion principle applied to the probability measure $p^i_G$ yields:
\begin{equation} \label{eq:third-3}
p^i_G(E) = p^i_G( E \cap U_{K-1}) + p^i_G(E \cap \mu_K) - p^i_G(E \cap U_{K-1} \cap \mu_K),
\end{equation}
so that:
\begin{equation}
p^i_G(E) = q_i \times \bigl[p^i_G(U_{K-1}) + p^i_G(\mu_K) - p^i_G(U_{K-1}\cap \mu_K)\bigr].
\end{equation}
Applying inclusion-exclusion again gives:
\begin{equation}
p^i_G(U_{K-1}) + p^i_G(\mu_K) - p^i_G(U_{K-1}\cap \mu_K) = p^i_G(G) = 1,
\end{equation}
and therefore $p^i_G(E) = q_i$.  If $U_{K-1} \cap \mu_K$ is empty, we run the same argument omitting Equations \ref{eq:third-1} and \ref{eq:third-2} and the third term on the right-hand side of Equation \ref{eq:third-3} (and carrying this omission through the next two equations).  This completes the induction step.
\end{proof}

We now have all the ingredients needed to state and prove our Agreement Theorem.

\begin{theorem} \label{thm:agreement}
Suppose the CPSs $(\mathcal G_A, p^A)$ and $(\mathcal G_B, p^B)$ satisfy certainty reflection and are $1$-closed.  Fix an event $E$, a state $\omega$, and numbers $q_A, q_B \in [0,1]$.  If it is common certainty at $\omega$ that Alice assigns probability $q_A$ to $E$ and Bob assigns probability $q_B$ to $E$, and local consistency holds at $\omega$, then $q_A = q_B$.
\end{theorem}

\begin{proof}
Assume $\omega \in C^\infty$, and let $H = m(\omega)$.  By Lemma \ref{lem:contain}, we have $H \subseteq C^\infty$.  From $H \in \mathcal G_A \wedge \mathcal G_B$, we have $H \in \mathcal G_A$, in particular, and therefore, by Part (i) of Lemma \ref{lem:ordered}:
\begin{equation}
H = \bigcup_{k=1}^K \mu_k,
\end{equation}
where $\mu_1, \dots, \mu_K$ are the distinct $\mathcal G_A$-atoms contained in $H$.  From $H \subseteq C^\infty \subseteq A_0$, every $x \in H$ satisfies $p^A_{m_A(x)}(E) = q_A$.  Therefore, each atom $\mu_k$ in the decomposition of $H$ satisfies $p^A_{\mu_k}(E) = q_A$.  By Lemma \ref{lem:average-atoms}, it follows that $p^A_H(E) = q_A$.  Running the same argument with $B$ in place of $A$, we get $p^B_H(E) = q_B$.  Since $H = m(\omega) \in \mathcal G_A \wedge\mathcal G_B$, local consistency at $\omega$ implies $p^A_H = p^B_H$ and therefore $q_A = q_B$.
\end{proof}

The next example shows that the conditions of our Agreement Theorem do not trivialize the framework.  It does not involve a common prior ($p^A_\Omega \not= p^B_\Omega$), both agents' conditioning families are strict subsets of the power set (minus the empty set), and the agreed-on probability of $E$ lies strictly between $0$ and $1$.

\begin{Ex} \label{ex:1-closed-cc}
Let $\Omega = \{a, b, c, d\}$ and set:
\begin{align}
\mathcal G_A &= \bigl\{\{a, b\}, \{c\}, \{d\}, \{c,d\},  \{a, b, c\}, \{a,b,d\}, \Omega\bigr\}, \\
\mathcal G_B &= \bigl\{\{c,d\}, \{a, b\},  \Omega \bigr\}.
\end{align}
Define CPSs as follows:
\begin{align}
&p^A_\Omega(\{a\}) = p^A_\Omega(\{b\}) = \dfrac{1}{2}, \\
&p^A_{\{c, d\}} (\{c\}) = p^A_{\{c,d\}}(\{d\}) =\frac{1}{2}, \\
&p^B_\Omega(\{c\}) = p^B_\Omega(\{d\}) = \dfrac{1}{2}, \\
&p^B_{\{a, b\}}(\{a\}) = p^B_{\{a, b\}} (\{b\})=\dfrac{1}{2},
\end{align}
It can be checked that both ${\cal G}_A$ and ${\cal G}_B$ satisfy certainty reflection and are $1$-closed, and that local consistency holds at states $a, b, c, d$.  

Let $E = \{a\}$ be the event of interest, $q_A = q_B = 1/2$, and calculate:
\begin{align}
A^0 = \{ a, b\}&, \,\,\,\, B^0 = \{a, b\}, \\
C_A(B^0) = \{ a, b\}&, \,\,\,\, C_B(A^0) = \{a,b\}, \\ 
A^1 = \{ a, b\}&, \,\,\,\, B^1 = \{a, b\},
\end{align}
from which the recursion stabilizes at:
\begin{equation}
A^n = \{ a, b\} \,\, \text{and} \,\, B^n = \{a, b\} \,\, \text{for all} \,\, n.
\end{equation}
It follows that $C^\infty = \{a, b\}$.  At state $a$ or $b$, it is common certainty that Alice and Bob each assign probability $1/2$ to $E$.  As promised, this example does not involve a common prior, the $1$-closed conditioning families are strict subsets of the power set, and the commonly-certain probability of $E$ lies in $(0, 1)$.
\end{Ex}

\section{1-Augmentation of CPSs} \label{sec:4}
The $1$-closedness property can be obtained via an augmentation process performed on a conditioning family.  Starting from a CPS $(\mathcal G, p)$, we enlarge $\mathcal G$ by adding in every event in the family $\mathcal E$ defined in Equation \ref{eq:augment}.  We then extend the CPS to the enlarged family.  The extension is obtained via the R\'enyi (1956) representation theorem for CPSs defined on conditioning families that are closed under finite unions.  We review this theorem below.  The details of extending a CPS are in the Appendix.  We go on to show that the $1$-augmentation operation stabilizes after finite many rounds and the resulting CPS is $1$-closed.  We illustrate this process by revisiting Example \ref{ex:disagree-cc} and showing that common certainty of disagreement disappears under augmentation.

\begin{Def} \label{def:1-augmentation}
The \textbf{1-augmentation} of $\mathcal G$, written $\widehat{\mathcal G}$, is the smallest family of subsets of $\Omega$ that contains $\mathcal E$, is closed under unions and nonempty intersections, and covers $\Omega$.
\end{Def}

We now summarize the R\'enyi (1956) representation theorem.  Fix a measurable space $(\Omega, \mathcal F)$.

\begin{Def} \label{def:dgen}
Let $(\Gamma, \prec)$ be a totally ordered set, and let $\{\mu_{\gamma}\}_{\gamma\in\Gamma}$ be measures on
$(\Omega, \mathcal F)$.  We call this family \textbf{dimensionally ordered} relative to $\mathcal G$ if, for every $G \in \mathcal G$,
there is a unique index $\gamma(G) \in \Gamma$ such that:
\begin{enumerate} [label=(\roman*)]
\item $0 < \mu_{\gamma(G)}(G) < +\infty$;
\item if $\alpha \prec \beta$ and $\mu_{\alpha}(G) < +\infty$, then
$\mu_{\beta}(G) = 0$.
\end{enumerate}
\end{Def}

It is readily checked that such a family generates a CPS by setting, for any event $E$:
\begin{equation} \label{eq:dgen-representation}
p_G(E) = \dfrac{\mu_{\gamma(G)}(E \cap G)}{\mu_{\gamma(G)}(G)}.
\end{equation}

\begin{Def}
A CPS defined from a dimensionally ordered family is called \textbf{dimensionally generated}.    
\end{Def}

The next result is Theorem 1 in R\'enyi (1956).  (In the case of finite $\Omega$ with $\mathcal G = 2^\Omega \backslash \emptyset$, Theorem 1 in Myerson, 1986 provides an equivalent representation by a sequence of full-support probability measures whose conditional probabilities converge pointwise.  We need the general case where $\mathcal G \not= 2^\Omega \backslash \emptyset$.)

\begin{theorem} \label{thm:renyi}
If $(\mathcal G, p)$ is a CPS and $\mathcal G$ is closed under finite unions, then it is dimensionally generated.
\end{theorem}

We show in the Appendix how to use this representation to construct an extension of a CPS to its $1$-augmented conditioning family.

\begin{theorem} \label{thm:big}
Let $(\mathcal G, p)$ be a CPS and let $\widehat{\mathcal G}$ be the $1$-augmentation of $\mathcal G$. Then:
\begin{enumerate} [label=(\roman*)]
\item there is a CPS $(\widehat{\mathcal G}, \widehat p)$ that extends $(\mathcal G, p)$, that is, $\widehat{p}_G = p_G$ for every $G \in \mathcal G$;
\item the iterated $1$-augmentation operation stabilizes after finitely many steps to produce a $1$-closed CPS.
\end{enumerate}
\end{theorem}

Condition (ii) of this result tell us how, starting from conditioning families for Alice and Bob that may not be $1$-closed, we can obtain $1$-closed extensions as employed in our Agreement Theorem.  We illustrate this process by revisiting Example \ref{ex:disagree-cc}.
\vspace{0.025in}

\paragraph{Example 1 contd.}
\textit{Applying Equation \ref{eq:augment}, we find:}
\begin{align}
\mathcal E_A &= \bigl\{\{b\}, \{d\}, \{a, d\}, \{b, c\}, \{b, d\}, \{c, d\}, \{a,b,d\}, \{a,c, d\}, \{b, c, d\}, \Omega\bigr\}, \\
\mathcal E_B &= \bigl\{\{c\}, \{d\}, \{a, c\}, \{a, d\}, \{b, c\}, \{b, d\}, \{c, d\}, \{a, b, c\}, \{a, b, d\}, \{a, c, d\}, \{b, c, d\}, \Omega\bigr\},
\end{align}
\textit{so that the $1$-augmentations are:}
\begin{align}
\widehat{\mathcal G}_A &= \bigl\{\{b\}, \{c\}, \{d\}, \{a, d\}, \{b, c\}, \{b, d\}, \{c, d\}, \{a,b,d\}, \{a,c, d\}, \{b, c, d\}, \Omega\bigr\}, \\
\widehat{\mathcal G}_B &= 2^\Omega \backslash \emptyset. \label{eq:54}
\end{align}

\textit{In this example, a single augmentation already produces $1$-closed CPS's.  For Bob this is immediate from Equation \ref{eq:54}.  For Alice, the only nonempty events omitted from $\widehat{\mathcal G}_A$ are:}
\begin{equation}
\{a\}, \,\, \{a, b\}, \,\, \{a, c\}, \,\, \{a, b, c\}.
\end{equation}
\textit{Every conditioning event in $\widehat{\mathcal G}_A$ containing one of these sets also contains $d$.  Since
$p^A_\Omega = \delta_d$ and $\widehat p_A$ extends $p_A$, the chain rule implies that every such conditional probability is also concentrated on $d$.  Consequently, none of the four omitted events receives conditional probability $1$.  It follows that no additional events are added by a second augmentation operation.}

\textit{To rerun the common-certainty recursion, we need only certain values of the augmented CPSs.  Since
$\{a, d\}, \{b, c\}\in\mathcal G_A$, the extension preserves:}
\begin{equation}
\widehat p^A_{\{a,d\}} = \delta_d, \,\, \widehat p^A_{\{b, c\}} = \delta_b.
\end{equation}
\textit{Since $\{b\}, \{c\}, \{d\} \in \widehat{\mathcal G}_A$ are singletons or were added as probability-$1$ event, we also have:}
\begin{equation}
\widehat p^A_{\{b\}} = \delta_b, \,\, \widehat p^A_{\{c\}}
= \delta_c, \,\, \widehat p^A_{\{d\}} = \delta_d.
\end{equation}
\textit{For Bob, every singleton $\{\omega\} \in \widehat{\mathcal G}_B$ and therefore:}
\begin{equation}
\widehat p^B_{\{\omega\}} = \delta_\omega.
\end{equation}
\textit{The augmented atoms are therefore:}
\begin{align}
&\widehat m_A(a) = \{a, d\}, \,\, \widehat m_A(b) = \{b\}, \,\, \widehat m_A(c) = \{c\}, \,\, \widehat m_A(d) = \{d\}, \\
&\widehat m_B(a) = \{a\}, \,\, \widehat m_B(b) = \{b\}, \,\, \widehat m_B(c) = \{c\}, \,\, \widehat m_B(d) = \{d\}.
\end{align}

\textit{Repeating the common-certainty recursion gives:}
\begin{align}
\widehat A^0 = \{b\},& \,\,\,\, \widehat B^0 = \{a, c, d\}, \\
\widehat C_A(\widehat B^0) = \{a, c, d\},& \,\,\,\, \widehat C_B(\widehat A^0) = \{b\}, \\
\widehat A^1 = \widehat A^0 \cap \widehat C_A(\widehat B^0) = \emptyset,& \,\,\,\, \widehat B^1 = \widehat B^0 \cap \widehat C_B(\widehat A^0) = \emptyset.
\end{align}
\textit{We see that $C^\infty = \emptyset$ and the previous scenario of common certainty of disagreement does not arise.  Observe that local consistency is not preserved under $1$-augmentation.  The set $\{b, c\}$ is an element of both $\widehat{\cal G}_A$ and $\widehat{\cal G}_B$.  However, we find:}
\begin{equation}
\widehat p^A_{\{b, c\}} = \delta_b \not= \delta_c = \widehat p^B_{\{b, c\}}.
\end{equation}
\textit{Conceptually, this is not surprising.  The $1$-augmentation process uncovers a disagreement that was invisible under the previous coarser information structure.  An implication is that, to obtain our Agreement Theorem, local consistency must be imposed on the $1$-augmented structure, not before.}

\section{Discussion} \label{sec:5}
We make some observations on the certainty-knowledge relationship, provide a characterization of the certainty reflection condition, and comment on open directions.
\vspace{0.1in}

\noindent\textbf{5.1 Knowledge vs.~Certainty}  In the usual set-up, with a single (prior) probability measure rather than a CPS, it is well known that certainty is weaker than knowledge, defined as follows.

\begin{Def} \label{def:knowledge}
Fix an agent $i = A, B$, an event $E$, and a state $\omega$.  Say $i$ \textbf{knows} $E$ at $\omega$ if $m_i(\omega) \subseteq E$.
\end{Def}

This remains true in our set-up, even after imposing certainty reflection and $1$-closedness.  Let $\Omega = \{a, b\}$, $\mathcal G = \{\{a\}, \Omega\}$, and define a CPS by setting $p_\Omega = \delta_{\{a\}}$.  It is easily verified that $1$-closedness and certainty reflection are satisfied.  Consider the event $E = \{a\}$.  At state $b$, the agent is certain of $E$, but does not know $E$.

The gap is closed if we add the assumption that $E$ happens, that is, $\omega \in E$.  We will prove the stronger result that, in this case, common certainty implies common knowledge.  To proceed, write the event that $i$ knows $E$ as:
\begin{equation}
K_i(E) = \{ \omega \in \Omega : m_i(\omega) \subseteq E\}.
\end{equation}
Define $A^0$ and $B^0$ as in Equation \ref{eq-4} and define recursively:
\begin{align}
A^{K, n+1} &= A^{K, n} \cap K_A(B^{K, n}), \\
B^{K, n+1} &= B^{K, n} \cap K_B(A^{K, n}),
\end{align}
for $n \ge 0$.  Let:
\begin{equation}
K^\infty = \bigcap_{n=0}^\infty A^{K, n} \cap \bigcap_{n=0}^\infty B^{K, n}.
\end{equation}

\begin{Def} \label{def:common-knowledge}
Fix an event $E$ and a state $\omega$.  Say it is \textbf{common knowledge at} $\omega$ that Alice assigns probability $q_A$ to $E$ and Bob assigns probability $q_B$ to $E$ if $\omega \in K^\infty$.
\end{Def}

Paralleling certainty, in the special case $q_A = q_B  = 1$, we simply say that the event $E$ is common knowledge at $\omega$.  

\begin{theorem} \label{thm:common-certainty-to-common-knowledge}
Suppose the CPSs $(\mathcal G_A, p^A)$ and $(\mathcal G_B, p^B)$ satisfy certainty reflection and are $1$-closed.  Fix an event $E$ and suppose $E$ is common certainty at state $\omega$ and $\omega \in E$.  Then $E$ is common knowledge at $\omega$.
\end{theorem}

\begin{proof}
Set $q_A = q_B = 1$ and let $J = C^\infty \cap E$.  We will prove by induction that $J \subseteq K^\infty$.  We first show that $J$ is $A$-saturated and $B$-saturated.  Fix $x \in J$, so that $x \in C^\infty$ and $x\in E$.  Since $C^\infty \subseteq A^0$, we have $x \in A^0$.  That is, $p^A_{m_A(x)}(E) = 1$, which, using $x \in E$ and Lemma \ref{lem:certainty-knowledge}, yields $m_A(x) \subseteq E$.  By Lemma \ref{lem:saturate-2}, the set $C^\infty$ is $A$-saturated.  Putting this together, we obtain:
\begin{equation}
m_A(x) \subseteq C^\infty \cap E = J,
\end{equation}
and therefore $J$ is $A$-saturated.  A parallel argument shows that $J$ is $B$-saturated.

We next show that $J \subseteq A^{K, n} \cap B^{K, n}$.  Lemma \ref{lem:certainty-knowledge} immediately yields $A^0 \cap E \subseteq A^{K, 0}$.  The parallel argument establishes that $B^0 \cap E \subseteq B^{K, 0}$.  Since $J \subseteq A^0 \cap B^0 \cap E$, this establishes the base step.  Now, assume inductively that $J \subseteq A^{K, n} \cap B^{K, n}$.  Fix $x \in J$.  Since $J$ is $A$-saturated, we have:
\begin{equation}
m_A(x) \subseteq J \subseteq B^{K, n},
\end{equation}
using the induction hypothesis.  This says that $x \in K_A(B^{K, n})$.  Again using the induction hypothesis, we now find:
\begin{equation}
x \in A^{K, n} \cap K_A(B^{K, n}) = A^{K, n+1}.
\end{equation}
The parallel argument shows that $x \in B^{K, n+1}$, completing the induction step.
\end{proof}

The converse inclusion $K^\infty \subseteq C^\infty$ follows immediately from the concentration property of CPSs alone.  This is a standard relationship.
\vspace{0.1in}

\noindent\textbf{5.2 Knowledge-Based Agreement}  
While our focus has been on certainty, our methods can also be used to obtain a knowledge-based Agreement Theorem for settings where agents do not have information partitions.  Geanakoplos (1989, 2021) and Samet (1990, 2022) are leading treatments of this case, relative to which our main innovation is again the CPS machinery.  The CPS-based theorem statement is: Suppose the CPSs $(\mathcal G_A, p^A)$ and $(\mathcal G_B, p^B)$ satisfy local consistency.  If, at a state $\omega$, it is common knowledge that Alice assigns probability $q_A$ to $E$ and Bob assigns probability $q_B$ to $E$, then $q_A = q_B$.  The first step in the proof is to show that $m(\omega) \subseteq K^\infty$.   (This is not novel.)  The next step is to use $K^\infty \subseteq C^\infty$ to obtain our Lemma \ref{lem:contain}.  From here, we piggyback on Lemmas \ref{lem:ordered} and \ref{lem:average-atoms} to obtain the result.
\vspace{0.1in}

\noindent\textbf{5.3 Certainty Reflection}
If the atoms $m_i(\omega)$ of a family $\mathcal G_i$ of conditioning events form a partition, it is immediate that certainty reflection (Definition \ref{def:reflection}) is satisfied.  The next result shows how, under certainty reflection, the departure from an atom-partition is controlled.  Under our closure conditions, the family of atoms fails to be a partition only if there are distinct states for which one atom is strictly contained in another atom.  We show that probability given the larger atom concentrates on the smaller atom.  In this sense, we obtain a partition up to null events.  In the theorem and proof we now give, the agent subscript is dropped for simplicity.

\begin{theorem} \label{thm:reflection}
Under $1$-closedness, certainty reflection holds if and only if, for every atom $m$ and state $\omega$, if $m \subsetneq m(\omega)$, then $p_{m(\omega)}(m) = 1$.
\end{theorem}

\begin{proof}
Suppose certainty reflection holds and let $m \subsetneq m(\omega)$.  Choose $\omega^\prime$ such that $m(\omega^\prime) = m$.  Let $E = m$ and $q = p_{m(\omega)}(E)$, and set:
\begin{equation}
F = \{\omega^{\prime\prime} \in \Omega : p_{m(\omega^{\prime\prime})}(E) = q\}.
\end{equation}
By certainty reflection, we have $p_{m(\omega)}(F) = 1$.  By concentration:
\begin{equation}
p_{m(\omega^\prime)}(E) = p_m(m) = 1.
\end{equation}
It follows that if $q \not= 1$, then $\omega^\prime \notin F$.  Since $p_{m(\omega)}(E) = q$, we have $\omega \in F$.  Therefore, by $p_{m(\omega)}(F) = 1$, $1$-closedness, and Lemma \ref{lem:certainty-knowledge}, we obtain $m(\omega)\subseteq F$.  This contradicts $\omega^\prime \notin F$.  Therefore, $q = 1$ and $p_{m(\omega)}(E) = p_{m(\omega)}(m) = 1$.

For the converse, fix an event $E$, a state $\omega^\prime$, and let $q = p_{m(\omega^\prime)}(E)$.  We need to show that $p_{m(\omega^\prime)}(F) = 1$.  Fix $\omega^{\prime\prime} \in m(\omega^\prime)$.  Then $m(\omega^{\prime\prime}) \subseteq m(\omega^\prime)$.  If there is equality, then:
\begin{equation}
p_{m(\omega^{\prime\prime})}(E) = p_{m(\omega^\prime)}(E) = q,
\end{equation}
and so $\omega^{\prime\prime} \in F$.  If instead $m(\omega^{\prime\prime}) \subsetneq m(\omega^\prime)$, then by assumption, we have $p_{m(\omega^\prime)}(m(\omega^{\prime\prime})) = 1$.  By the chain rule:
\begin{equation}
p_{m(\omega^\prime)}(E \cap m(\omega^{\prime\prime})) = p_{m(\omega^\prime)}(m(\omega^{\prime\prime})) \times p_{m(\omega^{\prime\prime})}(E) = p_{m(\omega^{\prime\prime})}(E).
\end{equation}
We also have:
\begin{equation}
p_{m(\omega^\prime)}(E \cap m(\omega^{\prime\prime})) = p_{m(\omega^\prime)}(E) = q.
\end{equation}
Therefore, $p_{m(\omega^{\prime\prime})}(E) = q$, from which $\omega^{\prime\prime} \in F$.  We have now shown that $m(\omega^\prime) \subseteq F$.  By concentration, it follows that $p_{m(\omega^\prime)}(F) = 1$, establishing certainty reflection.
\end{proof}

Certainty reflection is used in two places in the paper.  First, it is one of the hypotheses of our Agreement Theorem, appearing as a condition for Lemma \ref{lem:saturate-2}, which establishes saturation.  We do not know if our Agreement Theorem remains true or fails without this condition.  Next, certainty reflection is a hypothesis of Theorem \ref{thm:common-certainty-to-common-knowledge} above on when (common) certainty implies (common) knowledge.  We do not know if this result holds without certainty reflection.
\vspace{0.1in}

\noindent\textbf{5.4 The Infinite Case}  
Another question concerns the extension of our results beyond the finite case.  Nielsen (1984) and Samet (1990, 1992) pioneered the extension of the Agreement Theorem to the general measurable case.  In our setting, there are some existing pointers to an extension to the infinite setting.  The R\'enyi (1956) representation theorem holds in the general measurable case.  (Our review of it in Section \ref{sec:4} was at this level of generality.)  Kami\'nski (1985) studies various extensions of CPSs on general spaces and shows that they may need to be iterated transfinitely before stabilizing.  If topological structure is required, the methods in Doria (2025) and Doria and Yang (2025) yield CPSs defined on separable metric spaces, which may prove useful.  Basic to all ingredients of our Agreement Theorem, the finite setting allows us to operate directly on minimal conditioning events (atoms).  This strategy will need to be reworked in the infinite case.
\vspace{0.05in}

\noindent\textbf{5.5 Agreeing to Agree} Lehrer and Samet (2011) and Xiong (2012) study the question of when common knowledge of agreement is possible (i.e., when agents can ``agree to agree").  These papers treat the cases of finite, countable, and uncountable state spaces.  A CPS-based approach to this question might be insightful.

\renewcommand{\thesection}{A.1}
\setcounter{equation}{0}
\renewcommand\theequation{A.\arabic{equation}}
\setcounter{theorem}{0}
\renewcommand\thetheorem{A.\arabic{theorem}}
\setcounter{lemma}{0}
\renewcommand\thelemma{A.\arabic{lemma}}
\setcounter{Def}{0}
\renewcommand\thedefinition{A.\arabic{Def}}

\section*{Appendix} \label{app-a}
We prove Theorem \ref{thm:big} of Section \ref{sec:4}.  By Theorem \ref{thm:renyi}, there is a totally ordered set $\Gamma$ such that Equation \ref{eq:dgen-representation} holds.  Since $\mathcal G$ is finite, we may take $\Gamma$ to be finite without loss of generality.  Given $K \in \mathcal G$, we will refer to the index $\gamma(K)$ (recall that $0 < \gamma(K) <\infty$) as the active level of $K$.  

\begin{proof}[Proof of Theorem 3(i)]
Let $\mathcal A$ be the finite algebra generated by $\widehat{\mathcal G}$, and let $H_1, \dots, H_m$ be its atoms.  By construction, the atoms partition $\Omega$ and every member of $\widehat{\mathcal G}$ is a union of some of them.  Given an atom $H_i$, call $K \in \mathcal G$ an existing witness for $H_i$ if $H_i \subseteq K$ and $\mu_{\gamma(K)}(H_i) > 0$.  We first show that if $K, L\in \mathcal G$ are existing witnesses for a given atom $H_i$, then $\gamma(K) = \gamma(L)$.  Indeed, we have $H_i \subseteq K \cap L$, so $K \cap L \not= \emptyset$ from which $K \cap L \in \mathcal G$.  Also:
\begin{equation}
0< \mu_{\gamma(K)}(H_i) \le \mu_{\gamma(K)}(K \cap L) \le \mu_{\gamma(K)}(K) < +\infty.
\end{equation}
Therefore the active level of $K \cap L$ is $\gamma(K)$.  By the same argument it is also $\gamma(L)$, so that $\gamma(K) = \gamma(L)$.  This establishes that the active level of an existing witness depends only on the atom.

Next, we add an index $\gamma_0$ with $\gamma_0 \prec \gamma$ for all $\gamma \in \Gamma$, and set $\widehat{\Gamma} = \Gamma \cup \{\gamma_0\}$.  Let:
\begin{equation}
d_i =
\begin{cases}
\gamma(K) &\text{if} \,\, H_i \,\, \text{has an existing witness} \,\, K \in \mathcal G, \\
\gamma_0 &\text{otherwise}.
\end{cases}
\end{equation}
If $d_i \in \Gamma$, define the measure:
\begin{equation}
\lambda_i(E) = \mu_{d_i}(E \cap H_i),
\end{equation}
while, if $d_i = \gamma_0$, choose any probability measure $q_i$ with full support on $H_i$ and define:
\begin{equation}
\lambda_i(E) = q_i(E \cap H_i).
\end{equation}
Also, given an event $E$, let:
\begin{equation}
\nu_i(E) =
\begin{cases}
+\infty &\text{if} \,\, E \cap H_i \not= \emptyset, \\
0 &\text{otherwise}.
\end{cases}
\end{equation}

Now, for each $\eta \in \widehat{\Gamma}$, define a measure:
\begin{equation}
\widehat \mu_\eta(E) = \sum_{\{i \, : \, d_i = \eta\}} \lambda_i(E) + \sum_{\{i \, : \, d_i \succ \eta\}} \nu_i(E).
\end{equation}
Fix again $K \in \widehat{\mathcal G}$.  Since $K$ is a nonempty union of atoms, the set $\{d_i : H_i \subseteq K\}$ is nonempty and has a maximum, which we write as $\eta(K)$.  If $\alpha \prec \eta(K)$, choose $j$ with $H_j \subseteq K$ and $d_j = \eta(K)$.  Then $d_j \succ \alpha$, and so $\nu_j(K)$ appears in $\widehat \mu_\alpha(K)$.  Since $H_j \subseteq K$, we obtain $\nu_j(K ) = +\infty$.  Therefore, $\widehat{\mu}_\alpha(K) = +\infty$.

At the active level $\eta(K)$, no atom contained in $K$ has dimension strictly greater than $\eta(K)$, so that all $\nu$-terms vanish on $K$.  At the same time, there is a $j$ with $H_j \subseteq K$ and $d_j = \eta(K)$ such that $0 < \lambda_j(H_j) < +\infty$.  If $d_j \in \Gamma$, this follows from the definition of $d_j$ via an existing witness.  If $d_j = \gamma_0$, this follows from $\lambda_j(H_j) = q_j(H_j) = 1$.  Since only finitely many atoms are involved, we have shown that $0 < \widehat \mu_{\eta(K)}(K) < +\infty$.

Finally, if $\alpha \succ \eta(K)$, then no atom contained in $K$ has dimension equal to or greater than $\alpha$.  Since $K$ is a union of atoms, every atom not contained in $K$ is disjoint from $K$, from which every term in the definition of $\widehat \mu_\alpha(K)$ vanishes.  Therefore, $\widehat \mu_\alpha(K) = 0$.  We have now established that $\{\widehat \mu_\eta\}_{\eta \in \widehat \Gamma}$ is dimensionally ordered relative to $\widehat {\mathcal G}$.

It remains to show that $(\widehat{\mathcal G}, \widehat p)$ extends $(\mathcal G, p)$:  for $K \in \mathcal G$, we have $\eta(K) = \gamma(K)$ and, for any event $E$:
\begin{equation}
\widehat \mu_{\gamma(K)}(E \cap K) = \mu_{\gamma(K)}(E \cap K).
\end{equation}
Fix $K \in \mathcal G$ and write $\gamma = \gamma(K)$.  Since $0 < \mu_\gamma(K) < +\infty$ and $K$ is a finite union of the atoms it contains, there is an atom $H_j \subseteq K$ with $\mu_\gamma(H_j )> 0$.  Then $K$ is an existing witness for $H_j$, and so $d_j = \gamma$.  Therefore $\eta(K) \succeq \gamma$.

Now suppose there is an atom $H_i \subseteq K$ with $d_i = \delta \succ \gamma$.  By definition of $d_i$, there is a set $L \in \mathcal G$ such that $H_i \subseteq L$ and $\mu_\delta(H_i) > 0$.  Let $J = K \cap L$.  Then $J \in \mathcal G$, $J \not= \emptyset$, and:
\begin{equation}
0 < \mu_\delta(H_i) \le \mu_\delta(J) \le \mu_\delta(L) < +\infty,
\end{equation}
so that the active level of $J$ is $\delta$.  But $J \subseteq K$,  $0 < \mu_\gamma(K) < +\infty$, and $\gamma \prec \delta$.  By dimensional ordering applied to the set $K$, this implies $\mu_\delta(K) = 0$, which contradicts $H_i \subseteq K$ and $\mu_\delta(H_i) > 0$.  Therefore, no atom inside $K$ has dimension greater than $\gamma$.  It follows that $\eta(K) \preceq \gamma$ and therefore $\eta(K) = \gamma$.

Next, if $H_i \subseteq K$ and $d_i \not= \gamma$, then $\mu_\gamma(H_i) = 0$.  If not, then the set $K$ itself would be an existing witness for $H_i$.  This would imply $d_i = \gamma$, which is a contradiction.  Therefore the atoms inside $K$ with nonzero weight under $\mu_\gamma$ are exactly those with $d_i = \gamma$.  Since the atoms inside $K$ partition $K$, we get:
\begin{equation}
\mu_\gamma(E \cap K) = \sum_{\{i : H_i \subseteq K\}} \mu_\gamma(E \cap H_i) = \sum_{\{i : H_i \subseteq K \,\, \text{and} \,\, d_i = \gamma\}} \mu_\gamma(E \cap H_i) = \widehat \mu_\gamma(E \cap K).
\end{equation}
Setting $E = \Omega$ gives $\widehat \mu_\gamma(K)=\mu_\gamma(K)$.

We can now calculate, for $K \in \mathcal G$ and an event $E$:
\begin{equation}
\widehat p_K(E) = \dfrac{\widehat \mu_{\eta(K)}(E \cap K)}{\widehat \mu_{\eta(K)}(K)} = \dfrac{\widehat \mu_{\gamma(K)}(E\cap K)} {\widehat \mu_{\gamma(K)}(K)} = \dfrac{\mu_{\gamma(K)}(E \cap K)}{\mu_{\gamma(K)}(K)} =
p_K(E),
\end{equation}
establishing that $(\widehat{\mathcal G}, \widehat p)$ extends $(\mathcal G, p)$.
\end{proof}

\begin{proof}[Proof of Theorem 3(ii)]
Set $(\mathcal G^0, p^0) = (\mathcal G, p)$ and, for $n \ge 0$, define:
\begin{equation}
\mathcal E^n = \bigl\{ E \subseteq \Omega \, : \, p^n_G(E) = 1 \,\, \text{for some} \,\, G \in \mathcal G^n \,\, \text{with} \,\, E \subseteq G \bigr\}.
\end{equation}
Let $\mathcal G^{n+1}$ be the $1$-augmentation of $\mathcal G^n$, that is the smallest family containing $\mathcal E^n$ that covers $\Omega$ and is closed under unions and nonempty intersections.  By part (i), there is a CPS $(\mathcal G^{n+1}, p^{n+1})$ extending $(\mathcal G^n, p^n)$.  Concentration implies $\mathcal G^n \subseteq \mathcal E^n$, so that $\mathcal G^n \subseteq \mathcal G^{n+1}$ for all $n$.  If $(\mathcal G^n, p^n)$ is not $1$-closed, then $\mathcal E^n \not= \mathcal G^n$.  Since $\mathcal G^n \subseteq \mathcal E^n \subseteq \mathcal G^{n+1}$, it follows that $\mathcal G^n \subsetneq \mathcal G^{n+1}$.  Since $\Omega$ is finite, there is a finite $N$ with $\mathcal G^{N+1} = \mathcal G^N$, from which $\mathcal E^N = \mathcal G^N$.  It follows that $(\mathcal G^N, p^N)$ is $1$-closed, as required.
\end{proof}

\section*{Declaration of Generative AI and AI-Assisted Technologies in the Writing Process}

 ChatGPT and Grok were used to help find proof tactics and identify references.  After using these tools, the authors reviewed and edited the content as needed and take full responsibility for the content of the published article.

\end{document}